\documentclass[11pt]{article}
\UseRawInputEncoding

\usepackage{mathrsfs}
\usepackage{amssymb}
\usepackage{amsmath}
\usepackage{amsbsy}
\usepackage{epsfig}
\usepackage{enumerate}
\usepackage{bm}
\usepackage{color}
\usepackage[english]{babel}

\usepackage{bbm}
\usepackage{mathrsfs}
\usepackage{dsfont}
\usepackage{graphicx,theorem}

\makeatletter
\newcommand{\rmnum}[1]{\romannumeral #1}
\newcommand{\Rmnum}[1]{\expandafter\@slowromancap\romannumeral #1@}
\makeatother

\newtheorem{lem}{Lemma}[section]  \newtheorem{thm}{Theorem}[section]
\newtheorem{cor}{Corollary}[section] \newtheorem{defn}{Definition}[section]  
\numberwithin{equation}{section}
\newenvironment{proof}{{\noindent\it Proof}\quad}{\hfill $\square$\par}

\usepackage{soul}
\usepackage{color, xcolor}
\soulregister\cite7
\soulregister\ref7

 \newcommand{\me}{\mathrm{e}} 
\newcommand{\dif}{\mathrm{d}} \DeclareMathAlphabet{\mathsfsl}{OT1}{cmss}{m}{sl} \DeclareMathAlphabet{\mathpzc}{OT1}{pzc}{m}{it}

    \newcommand{\ee}{\mathbb{E}}
   \newcommand{\nn}{\mathbb{N}} \newcommand{\rr}{\mathbb{R}}
    
\newcommand{\vv}{\mathbb{V}}
 \def\CC{\mathcal C}   \def\FF{\mathcal F}  \def\HH{\mathcal H}

\def\d"{^{\prime\prime}} \def\d'{^{\prime}}

\begin{document}

	\begin{center}{\LARGE\bf Complete moment convergence of moving average processes for $m$-widely acceptable sequence under sub-linear expectations}
	\end{center}
	
	%\title[]{}
	%\thanks{Project supported by Doctoral Scientific Research Starting Foundation of Jingdezhen Ceramic Institute (Nos.102/01003002031), Scientific Program of Department of Education of Jiangxi Province of China (Nos. GJJ190732, GJJ180737), Natural Science Foundation Program of Jiangxi Province 20202BABL211005, National Natural Science Foundation of China (Nos. 61662037)}
	%\subjclass[2000]{60F15, 60F05}

	%\keywords{Complete moment convergence; Capacity; I. i. d. random variables; Weighted sums; Sublinear expectation}
	%\date{} \maketitle
	
	\begin{center}
		Mingzhou Xu~*\footnote{*correspondence: mingzhouxu@whu.edu.cn} \quad Xuhang Kong~\footnote{Email: 869458367@qq.com} \quad 
		\\
		School of Information Engineering, Jingdezhen Ceramic University\\
		Jingdezhen 333403, China
	\end{center}
	
	\renewcommand{\abstractname}{~}
	\begin{abstract}
		{\bf Abstract:}
		In this article, the complete moment convergence for the partial sum of moving average processes $\{X_n=\sum_{i=-\infty}^{\infty}a_iY_{i+n},n\ge 1\}$ is estabished under some proper conditions, where $\{Y_i,-\infty<i<\infty\}$ is a sequence of $m$-widely acceptable ($m$-WA) random variables, which is stochastically dominated by a random variable $Y$ in sub-linear expectations space $(\Omega,\HH,\ee)$ and $\{a_i,-\infty<i<\infty\}$ is an absolutely summable sequence of real numbers. The results extend the relevant results in probability space to those under sub-linear expectations.
		
		
		\vspace{-3mm}
	\end{abstract}

	\section{Introduction }
	In order to study the uncertainty in probability, Peng \cite{Peng2007,Peng2010,Peng2019} introduced the concepts of the sub-linear expectations space. Motivated by the works of Peng \cite{Peng2007,Peng2010,Peng2019}, lots of people try to extend the results of classic probability space to those of the sub-linear expectations space. Zhang \cite{Zhang2016a,Zhang2016b,Zhang2015} got the exponential inequalities, Rosenthal's inequalities, and Donsker's invariance principle under sub-linear expectations. Under sub-linear expectations, Xu and Cheng \cite{Xu2022a} studied  how small the increments of $G$-Brownian motion are. Xu and Zhang \cite{Xujiapan2019,Xujiapan2020} got a three series theorem of independent random variables and  a law of logarithm for arrays of row-wise extended negatively dependent random variables under the sub-linear expectations. Zhong and Wu \cite{Zhong2017} obtained  the complete convergence and complete moment convergence for weighted sums of extended negatively dependent random variables under sub-linear expectations. For more limit theorems under sub-linear expectations, the interested readers could refer to Wu and Jiang \cite{Wuqunying2018}, Zhang and Lin \cite{Zhang2018}, Zhong and Wu \cite{Zhong2017}, Hu and Yang \cite{Huzechun2017}, Chen \cite{Chen2016},   Zhang \cite{Zhang2016c},  Hu, Chen, and Zhang \cite{Hufeng2014}, Gao and Xu \cite{Gao2011},  Kuczmaszewska \cite{Kuczm}, Chen and Wu \cite{Chen2022}, Xu and Cheng \cite{Xu2021b,Xu2022a}, Xu et al. \cite{Xu2022b,Xu2023b}, Xu and Kong \cite{Xu2023}, and references therein.
	
	Guan, Xiao and Zhao \cite{guan2021complete} studied complete moment convergence of moving average processes for $m$-WOD sequence. For more results about complete moment convergence of moving average processes, the interested reader could refer to Zhang and Ding \cite{Zhang-Ding2017}, Hosseini and Nezakati \cite{Hosseni2020} and refercences therein. The main conclusions of Guan, Xiao and Zhao \cite{guan2021complete} are that under proper conditions the complete moment convergence for the partial sum of moving average processes produced by $m$-widely orthant dependent random variables holds. Recently, Wu, Deng, and Wang studied capacity inequalities and strong laws for $m$-widely acceptable ($m$-WA) random variables under sub-linear expectations. It is natural to wonder whether or not the relevant results of  Guan, Xiao and Zhao \cite{guan2021complete} hold for moving average processes produced by $m$-WA random variables under sub-linear expectations. Here, we try to get the complete moment convergence for the partial sum of moving average processes produced by $m$-WA random variables under sub-linear expectations, complementing the relevant results obtained in Guan, Xiao and Zhao \cite{guan2021complete}.
	
	We organize the rest of this paper as follows. We give some necessary basic notions, concepts and corresponding  properties, and cite necessary lemma under sub-linear expectations in the next section.  In Section 3, we give our main results, Theorems \ref{thm01}-\ref{thm02}, the proofs of which are also presented in this section.

	\section{Preliminaries}
	As in Xu and Cheng \cite{Xu2021b}, we use similar notations as in the work by Peng \cite{Peng2010,Peng2019}, Chen \cite{Chen2016}, and Zhang \cite{Zhang2016b}. Suppose that $(\Omega,\FF)$ is a given measurable space. Assume that $\HH$ is a subset of all random variables on $(\Omega,\FF)$ such that $X_1,\cdots,X_n\in \HH$ implies $\varphi(X_1,\cdots,X_n)\in \HH$ for each $\varphi\in \CC_{l,Lip}(\rr^n)$, where $\CC_{l,Lip}(\rr^n)$ represents the linear space of (local lipschitz) function $\varphi$ fulfilling
	$$
	|\varphi(\mathbf{x})-\varphi(\mathbf{y})|\le C(1+|\mathbf{x}|^m+|\mathbf{y}|^m)(|\mathbf{x}-\mathbf{y}|), \forall \mathbf{x},\mathbf{y}\in \rr^n
	$$
	for some $C>0$, $m\in \nn$ depending on $\varphi$.
	\begin{defn}\label{defn1} A sub-linear expectation $\ee$ on $\HH$ is a functional $\ee:\HH\mapsto \bar{\rr}:=[-\infty,\infty]$ fulfilling the following properties: for all $X,Y\in \HH$, we have
		\begin{description}
			\item[\rm (a)]  Monotonicity: If $X\ge Y$, then $\ee[X]\ge \ee[Y]$;
			\item[\rm (b)] Constant preserving: $\ee[c]=c$, $\forall c\in\rr$;
			\item[\rm (c)] Positive homogeneity: $\ee[\lambda X]=\lambda\ee[X]$, $\forall \lambda\ge 0$;
			\item[\rm (d)] Sub-additivity: $\ee[X+Y]\le \ee[X]+\ee[Y]$ whenever $\ee[X]+\ee[Y]$ is not of the form $+\infty-\infty$ or $-\infty+\infty$.
		\end{description}
		
	\end{defn}
	
	A set function $V:\FF\mapsto[0,1]$ is named to be a capacity if
	\begin{description}
		\item[\rm (a)]$V(\emptyset)=0$, $V(\Omega)=1$;
		\item[\rm (b)]$V(A)\le V(B)$, $A\subset B$, $A,B\in \FF$.\\
		%Moreover, if $V$ is continuous, then $V$ should obey
		%\item[\rm (c)] $V(A_n)\uparrow V(A)$, if $A_n\uparrow A$.
		%\item[\rm (d)] $V(A_n)\downarrow V(A)$, if $A_n\downarrow A$.
	\end{description}
	A capacity $V$ is called sub-additive if $V(A\bigcup B)\le V(A)+V(B)$, $A,B\in \FF$.
	
	In this sequel, given a sub-linear expectation space $(\Omega, \HH, \ee)$, set $\vv(A):=\inf\{\ee[\xi]:I_A\le \xi, \xi\in \HH\}=\ee[I_A]$, $\forall A\in \FF$ (see (2.3) and the definitions of $\vv$ above (2.3) in Zhang \cite{Zhang2016a}). $\vv$ is a sub-additive capacity. Set
	$$
	C_{\vv}(X):=\int_{0}^{\infty}\vv(X>x)\dif x +\int_{-\infty}^{0}(\vv(X>x)-1)\dif x.
	$$
	As in 4.3 of Zhang \cite{Zhang2016a}, throughout this paper, define an extension of $\ee$ on the space of all random variables by
	$$
	\ee^{*}(X)=\inf\left\{\ee[Y]:X\le Y,Y\in\HH\right\}.
	$$
	Then $\ee^{*}$ is a sublinear expectation on the space of all random variables, $\ee[X]=\ee^{*}[X]$, $\forall X\in \HH$, and $\vv(A)=\ee^{*}(I_A)$, $\forall A\in \FF$.
	
	\begin{defn}\label{defn02}
		Suppose $\{Y_n,n\ge 1\}$ is a sequence of random variables in sub-linear expectations space $(\Omega,\HH,\ee)$. $\{Y_n,n\ge 1\}$ is called to be widely acceptable (WA), if there exists a positive sequence $\{g(n),n\ge 1\}$ of dominating coefficients  such that for all $n\in \mathbb{N}$, we have
		\begin{equation}\label{01}
			\ee\exp\left(\sum_{i=1}^{n}a_{ni}f_i(Y_i)\right)\le g(n)\prod_{i=1}^{n}\ee\exp\left(a_{ni}f_i(Y_i)\right), \quad 0<g(n)<\infty,
		\end{equation}
		where $\{a_{ni},1\le i\le n,n\ge1\}$ is an array of nonnegative constants and $f_i(\cdot)\in C_{b,Lip}(\rr)$, $i=1,2,\ldots$, are all non-decreasing ( or all non-increasing ) real valued truncation functions.
	\end{defn}
	\begin{defn}\label{defn03}
		Let $m\ge 1$ be fixed integer. A sequence of random variables $\{X_n,n\ge 1\}$ is called to be $m$-wildely acceptable ($m$-WA), if for any $n\ge 2$, and $i_1,i_2,\cdots,i_n$ fulfilling $|i_k-i_j|\ge m$ for all $1\le k\not= j\le n$, we have $X_{i_1}, X_{i_2}, \cdots, X_{i_n}$ are WA.
	\end{defn}
	\begin{defn}\label{defn-stoch} We say that $\{Y_n;n\ge 1\}$ is stochastically dominated by a random variable $Y$ in $(\Omega, \HH, \ee)$, if there exists a constant $C$ such that $\forall n\ge 1$, for all non-negative $h\in \CC_{l,Lip}(\rr)$, $\ee(h(Y_n))\le C \ee(h(Y)).$	
	\end{defn}
	
	Assume that $\mathbf{X}_1$ and $\mathbf{X}_2$ are two $n$-dimensional random vectors defined, respectively, in sub-linear expectation spaces $(\Omega_1,\HH_1,\ee_1)$ and $(\Omega_2,\HH_2,\ee_2)$. They are called identically distributed if for every function $\psi\in\CC_{l,Lip}(\rr)$ such that $\psi(\mathbf{X}_1)\in \HH_1, \psi(\mathbf{X}_2)\in \HH_2$,
	$$
	\ee_1[\psi(\mathbf{X}_1)]=\ee_2[\psi(\mathbf{X}_2)], \mbox{  }
	$$
	whenever the sub-linear expectations are finite. $\{X_n\}_{n=1}^{\infty}$ is named to be identically distributed if for each $i\ge 1$, $X_i$ and $X_1$ are identically distributed.
	
	In the paper we assume that $\ee$ is countably sub-additive, i.e., $\ee(X)\le \sum_{n=1}^{\infty}\ee(X_n)$, whenever $X\le \sum_{n=1}^{\infty}X_n$, $X,X_n\in \HH$, and $X\ge 0$, $X_n\ge 0$, $n=1,2,\ldots$. Hence $\ee^{*}$ is also countably sub-additive. Let $C$ stand for a positive constant which may change from place to place. $I(A)$ or $I_A$ represent the indicator function of $A$. Write $\log (x)=\ln\max\{\me, x\}$, $x>0$.

	We cite the following lemma (cf. Lemma 2.2 of Xu et al. \cite{Xu2023b}).
	\begin{lem}\label{lem01+}If for a random variable  $X$ on $(\Omega,\FF)$, $C_{\vv}\{|X|\}<\infty$, then
		\[
		\ee^{*}[|X|]\le C_{\vv}\{|X|\}.
		\]
	\end{lem}
	
	Next we cite and give some useful lemmas.
	\begin{lem}\label{lem01}(cf. Wu et al. \cite{wu2023capacity})Let $\{X_n,n\ge 1\}$ be a sequence of $m$-WA random variables with dominating coefficients $g(n)$. If $\{f_n(\cdot),n\ge 1\}$ are all non-decreasing (non-increasing), then $\{f_n(X_n),n\ge 1\}$ are still $m$-WA with dominating coefficients $\{g(n),n\ge 1\}$.	
	\end{lem}
	\begin{lem}\label{lem02}
		Let $0<t\le 1$ or $t=2$ and $\{X_n,n\ge 1\}$ be a sequence of WA random variables in sub-linear expectations space $(\Omega, \HH, \ee)$. Assume further that $\ee(X_n)\le 0$ for each $n\ge1$ when $t=2$.  Then for all $x>0$, and $y>0$,
		\begin{equation}\label{02}
			\vv\left(S_n\ge x\right)\le \sum_{i=1}^{n}\vv\left(X_i> y\right)+g(n)\exp\left(\frac{x}{y}-\frac{x}{y}\ln\left(1+\frac{xy^{t-1}}{\sum_{i=1}^{n}\ee|X_i|^t}\right)\right).
		\end{equation}
	\end{lem}
	\begin{proof}
		If $0<t\le1$, then we can establish (\ref{02}) by the adapted proof of Theorem 2.1 of Shen \cite{shen2011probability}. If $t=2$, (\ref{02}) follows immediately from Lemma 2.1 of Wu et al. \cite{wu2023capacity}. For readers' convenience, here we give detailed proof when $0<t\le 1$.
		
		For $y>0$, write $\bar{X}_i=\min\{X_i,y\}$, $i=1,2,\cdots, n$, and $T_n=\sum_{i=1}^{n}\bar{X}_i$, $n\ge 1$. We easily see that 
		$$
		\left\{S_n\ge x\right\}=\left\{T_n\not= S_n\right\}\bigcup\left\{T_n\ge x\right\},
		$$
		which yields that for any positive $h$,
		\[\vv\left(S_n\ge x\right)\le \vv\left(T_n\not= S_n\right)+\vv\left(T_n\ge x\right)\le \sum_{i=1}^{n}\vv\left(X_i> y\right)+\me^{-hx}\ee\me^{hT_n}.
		\]
		It follows that 
		\begin{equation}\label{03}
			\vv\left(S_n\ge x\right)\le\sum_{i=1}^{n}\vv\left(X_i> y\right)+g(n)\me^{-hx}\prod_{i=1}^{n}\ee\me^{h\bar{X}_i}.
		\end{equation}
		For $0<t\le 1$, $h>0$, the function $\frac{\me^{hu}-1}{u^t}$ is increasing on $u>0$. Hence
		\begin{eqnarray*}
			\ee\me^{h\bar{X}_i}&\le&1+\ee\left(\frac{\me^{h\bar{X}_i}-1}{|\bar{X}_i|^t}|\bar{X}_i|^t\right)\le 1+\ee\left(\frac{\me^{hy}-1}{|y|^t}|\bar{X}_i|^t\right)\\
			&\le&1+\frac{\me^{hy}-1}{|y|^t}\ee\left(|\bar{X}_i|^t\right)\le \exp\left\{\frac{\me^{hy}-1}{|y|^t}\ee\left(|\bar{X}_i|^t\right)\right\}\\
			&\le&\exp\left\{\frac{\me^{hy}-1}{|y|^t}\ee\left(|X_i|^t\right)\right\}.
		\end{eqnarray*}
		Combining the inequality above and (\ref{03}) yields that 
		\begin{equation}\label{04}
			\vv\left(S_n\ge x\right)\le\sum_{i=1}^{n}\vv\left(X_i> y\right)+g(n)\exp\left\{\frac{\me^{hy}-1}{y^t}\sum_{i=1}^{n}\ee\left(|X_i|^t\right)-hx\right\}.
		\end{equation}
		Taking $h=\frac{1}{y}\log\left(1+\frac{xy^{t-1}}{\sum_{i=1}^{n}\ee\left(|X_i|^t\right)}\right)$ in the right-hand side of (\ref{04}), we obtain (\ref{02}).
	\end{proof}
	
	\begin{lem}\label{lem03}
		For a positive real number $q\ge 2$, if $\{X_n,n\ge 1\}$ is a sequence of $m$-WA random variables with dominating coefficients $\{g(n),n\ge 1\}$. If $C_{\vv}\left\{|X_i|^q\right\}<\infty$ for every $i\ge 1$, then for all $n\ge 1$, there exist positive constants $C_1(m,q)$, $C_2(m,q)$, and $C_3(m,q)$ depending on $q$ and $m$ such that 
		\begin{eqnarray*}
			&&\ee\left(\left|\sum_{i=1}^{n}X_i\right|^q\right)\le  C_1(m,q)\sum_{i=1}^{n}C_{\vv}\left\{|X_i|^q\right\}\\
			&&\quad+C_2(m,q)g(n)\left(\sum_{i=1}^{n}\ee X_i^2\right)^{q/2}+C_3(m,q)\left(\sum_{i=1}^{n}\left[|\ee(X_i)|+|\ee(-X_i)|\right]\right)^q.
		\end{eqnarray*}
	\end{lem}
	%\begin{eqnarray*}
	%&&|X|=\lim_{n\rightarrow\infty}\sum_{k=1}^{n}|X|I\{k-1<|X|\le k\}\\
	%&&\quad\ge \lim_{n\rightarrow\infty}\sum_{k=1}^{n}(k-1)I\{k-1<|X|\le k\}=\lim_{n\rightarrow\infty}\sum_{k=1}^{n}(k-1)\left(I\{|X|>k-1\}-I\{|X|>k\}\right)\\
	%&&\quad\ge \lim_{n\rightarrow\infty}\sum_{k=1}^{n-1}I\{|X|>k\}-(n-1)I\{|X|>n\},
	%\end{eqnarray*}
	\begin{proof}
		Note that 
		\[C_{\vv}\{|X^{+}|^q\}=\int_{0}^{\infty}\vv\left(|X^{+}|^q>x\right)\dif x=\int_{0}^{\infty}qx^{q-1}\vv\left\{|X^{+}|>x\right\}\dif x,\] where $X^{+}:=\max\{X,0\}$. We first suppose that   $\{X_n,n\ge 1\}$ is a sequence of WA random variables with dominating coefficients $\{g(n),n\ge 1\}$ and $\ee(X_n)\le 0$. Putting $y=x/r$ in (\ref{03}) yields 
		\begin{equation}\label{05}
			\vv\left(S_n^{+}\ge x\right)\le \sum_{i=1}^{n}\vv\left(X_i^{+}> y\right)+g(n)\me^r\left(\frac{r\sum_{i=1}^{n}\ee\left(|X_i|^2\right)}{r\sum_{i=1}^{n}\ee\left(|X_i|^2\right)+x^2}\right)^r.
		\end{equation}
		By the similar proof of (3.4) of Zhang \cite{Zhang2016c},  $n\ge 1$, multiplying both sides of (\ref{02}) by $qx^{q-1}$, and integrating on the half line, we have 
		\begin{eqnarray*}
			&&\ee\left(\left(S_n^{+}\right)^q\right)\le C_{\vv}\left\{\left(S_n^{+}\right)^q\right\}\\
			&&\quad \le \sum_{i=1}^{n}CC_{\vv}\left\{|X_i^{+}|^q\right\}+Cg(n)\left(\sum_{i=1}^{n}\ee X_i^2\right)^{q/2}.
		\end{eqnarray*}
		Hence when $\{X_n,n\ge 1\}$ is a sequence of WA random variables with dominating coefficients $\{g(n),n\ge 1\}$, by $C_r$ inequality, we see that 
		\begin{eqnarray*}
			&&\ee\left(\left(S_n^{+}\right)^q\right)\le C\ee\left(\left((S_n-\sum_{i=1}^{n}\ee X_i)^{+}\right)^q\right)+C\left(\sum_{i=1}^{n}|\ee(X_i)|\right)^q\\
			&&\quad \le \sum_{i=1}^{n}CC_{\vv}\left\{|(X_i-\ee X_i)^{+}|^q\right\}+Cg(n)\left(\sum_{i=1}^{n}\ee X_i^2\right)^{q/2}++C\left(\sum_{i=1}^{n}|\ee(X_i)|\right)^q\\
			&&\quad \le . \sum_{i=1}^{n}CC_{\vv}\left\{|X_i^{+}|^q\right\}+C\sum_{i=1}^{n}|\ee(X_i)|^q+Cg(n)\left(\sum_{i=1}^{n}\ee X_i^2\right)^{q/2}++C\left(\sum_{i=1}^{n}|\ee(X_i)|\right)^q\\
			&&\quad \le \sum_{i=1}^{n}CC_{\vv}\left\{|X_i^{+}|^q\right\}+Cg(n)\left(\sum_{i=1}^{n}\ee X_i^2\right)^{q/2}++C\left(\sum_{i=1}^{n}|\ee(X_i)|\right)^q,
		\end{eqnarray*}
		and 
		\begin{eqnarray*}
			&&\ee\left(\left((-S_n)^{+}\right)^q\right) \le \sum_{i=1}^{n}CC_{\vv}\left\{|X_i^{-}|^q\right\}+Cg(n)\left(\sum_{i=1}^{n}\ee X_i^2\right)^{q/2}++C\left(\sum_{i=1}^{n}|\ee(-X_i)|\right)^q.
		\end{eqnarray*}
		Therefore, combining the two equations above yields 
		\begin{eqnarray*}
			&&\ee\left(\left|\sum_{i=1}^{n}X_i\right|^q\right)\le  C_1(m,q)\sum_{i=1}^{n}C_{\vv}\left\{|X_i|^q\right\}\\
			&&\quad+C_2(m,q)g(n)\left(\sum_{i=1}^{n}\ee X_i^2\right)^{q/2}+C_3(m,q)\left(\sum_{i=1}^{n}\left[|\ee(X_i)|+|\ee(-X_i)|\right]\right)^q.
		\end{eqnarray*}
		When $\{X_n,n\ge 1\}$ is a sequence of $m$-WA random variables with dominating coefficients $\{g(n),n\ge 1\}$, by the equation above, the similar proof of Corollary 3 of Fang et al. \cite{fang2020asymptotic} ( or the adapted proof of Theorem 2.2 of Wu et al. \cite{wu2023capacity} ) and $C_r$ inequality, we finish the proof of this theorem.
	\end{proof}

	\section{Main results}
	Our main results, considered as an extension of Guan et al. \cite{guan2021complete} in some sense, are as follows.
	\begin{thm}\label{thm01}
		Suppose $l(x)$ is a function slowly varying at infinity, $p\ge 1$, $\alpha>\frac12$, $\alpha  p>1$. Suppose that $\{a_i,-\infty<i<\infty\}$ is an absolutely summable sequence of real numbers. Assume that $\{X_n=\sum_{i=-\infty}^{\infty}a_iY_{i+n},n\ge 1\}$ is a moving average process produced by a sequence $\{Y_i, -\infty<i<\infty\}$ of $m$-WA random variables with dominating coefficients $g(n)=O(n^{\delta})$ for some $\delta\ge 0$, and $\{Y_i,-\infty<i<\infty\}$ is stochastic dominated by $Y$ in sub-linear expectation space $(\Omega,\HH,\ee)$. If $\ee(Y_i)=\ee(-Y_i)=0$, $i=1,2,\cdots$, for $\frac12<\alpha\le 1$, $C_{\vv}\left\{|Y|^pl(|Y|^{1/\alpha})\right\}<\infty$ for $p>1$ and $C_{\vv}\left\{|Y|^{1+\lambda}\right\}<\infty$ for $p=1$ and some $\lambda>0$, then for any $\epsilon>0$,
		\begin{equation}\label{3.1}
			\sum_{n=1}^{\infty}n^{\alpha p-2-\alpha}l(n)C_{\vv}\left\{\left(\left|\sum_{j=1}^{n}X_j\right|-\epsilon n^{\alpha}\right)^{+}\right\}<\infty.
		\end{equation}
	\end{thm}
	\begin{proof}
		For $2^{-\alpha}<\mu<1$, define $\tilde{g}_{\mu}(x)\in \CC_{l,Lip}(\rr)$ such that $I\{|x|\le \mu\}<\tilde{g}_{\mu}(x)< I\{|x|\le 1\}$.
		Define $g_j(x)\in \CC_{l,Lip}(\rr)$, $j\ge 1$ such that $g_j(x)$ is even function, and for $x$, $0\le g_j(x)\le 1$; $g_j(x/2^{j\alpha})=1$ while $2^{(j-1)\alpha}<|x|\le 2^{j\alpha}$, and $g_j(x/2^{j\alpha})=0$ while $|x|\le \mu 2^{(j-1)\alpha}$ or $|x|>(1+\mu)2^{j\alpha}$.
		We see that 
		\begin{equation}\label{001}
			g_j(|Y|/2^{j\alpha})\le I\left\{\mu 2^{(j-1)\alpha}<|Y|\le (1+\mu)2^{j\alpha}\right\}, |Y|^l\tilde{g}_{\mu}\left(\frac{|Y|}{2^{k\alpha}}\right)\le 1+\sum_{j=1}^{k}|Y|^lg_j\left(\frac{|Y|}{2^{j\alpha}}\right),
		\end{equation}
		\begin{equation}\label{002}
			1-\tilde{g}_{\mu}\left(\frac{|Y|}{2^{k\alpha}}\right)\le \sum_{j=k}^{\infty}g_j\left(\frac{|Y|}{2^{j\alpha}}\right).
		\end{equation}
		
		Let $f(n)=n^{\alpha p-2-\alpha}l(n)$, $Y_{xj}^{(1)}=-xI\{Y_j<-x\}+Y_jI\{|Y_j|\le x\}+xI\{Y_j>x\}$,
		$Y_{xj}^{(2)}=Y_j-Y_{xj}^{(1)}$ be the monotone trunctions of $\{Y_j,-\infty<j<\infty\}$ for $x>0$. Write $Y_x^{(1)}=-xI\{Y<-x\}+YI\{|Y|\le x\}+xI\{Y>x\}$,  $Y_{x}^{(2)}=Y-Y_x^{(1)}$. Then by Lemma \ref{lem01}, we see that $\{Y_{xj}^{(1)},-\infty<j<\infty\}$ and $\{Y_{xj}^{(2)},-\infty<j<\infty\}$ are two sequences of $m$-WA random variables. We observe that
		\begin{eqnarray}\label{3.2}
			\nonumber	&&\sum_{n=1}^{\infty}f(n)C_{\vv}\left\{\left(\left|\sum_{j=1}^{n}X_j\right|-\epsilon n^{\alpha}\right)^{+}\right\}\\
			\nonumber&&\quad\le \sum_{n=1}^{\infty}f(n)\int_{\epsilon n^{\alpha}}^{\infty}\vv\left\{\left|\sum_{j=1}^{n}X_j\right|>x\right\}\dif x\le C\sum_{n=1}^{\infty}f(n)\int_{ n^{\alpha}}^{\infty}\vv\left\{\left|\sum_{j=1}^{n}X_j\right|>\epsilon x\right\}\dif x\\
			\nonumber&&\quad\le C\sum_{n=1}^{\infty}f(n)\int_{ n^{\alpha}}^{\infty}\vv\left\{\left|\sum_{i=-\infty}^{\infty}a_i\sum_{j=i+1}^{i+n}Y_{xj}^{(2)}\right|>\epsilon x/2\right\}\dif x\\
			&&\quad\quad +C\sum_{n=1}^{\infty}f(n)\int_{ n^{\alpha}}^{\infty}\vv\left\{\left|\sum_{i=-\infty}^{\infty}a_i\sum_{j=i+1}^{i+n}Y_{xj}^{(1)}\right|>\epsilon x/2\right\}\dif x=:\Rmnum{1}_1+\Rmnum{1}_2.
		\end{eqnarray}
		Firstly, we establish $\Rmnum{1}_1<\infty$. Observe $|Y_{xj}^{(2)}|<|Y_j|\left(1-\tilde{g}_{\mu}\left(\frac{|Y_j|}{x}\right)\right)$. Then by Markov's inequality under sub-linear expectations, we see that 
		\begin{eqnarray*}
			\Rmnum{1}_1&\le& C\sum_{n=1}^{\infty}f(n)\int_{ n^{\alpha}}^{\infty}x^{-1}\ee^{*}\left|\sum_{i=-\infty}^{\infty}a_i\sum_{j=i+1}^{i+n}Y_{xj}^{(2)}\right|\dif x\\
			&\le&C\sum_{n=1}^{\infty}f(n)\int_{ n^{\alpha}}^{\infty}x^{-1}\sum_{i=-\infty}^{\infty}|a_i|\sum_{j=i+1}^{i+n}\ee^{*}\left|Y_{xj}^{(2)}\right|\dif x\\
			&\le&C\sum_{n=1}^{\infty}f(n)\int_{ n^{\alpha}}^{\infty}x^{-1}\ee|Y|\left(1-\tilde{g}_{\mu}\left(\frac{|Y|}{x}\right)\right)\dif x\\
			&\le&C\sum_{n=1}^{\infty}f(n)\sum_{m=n}^{\infty}\int_{ m^{\alpha}}^{(m+1)^{\alpha}}x^{-1}\ee|Y|\left(1-\tilde{g}_{\mu}\left(\frac{|Y|}{m^{\alpha}}\right)\right)\dif x\\
			&\le&C\sum_{m=1}^{\infty}m^{-1}\ee|Y|\left(1-\tilde{g}_{\mu}\left(\frac{|Y|}{m^{\alpha}}\right)\right)\sum_{n=1}^{m}f(n).
		\end{eqnarray*}
		If $p>1$, $\alpha p-1-\alpha>-1$, we conclude that
		\begin{eqnarray*}
			\Rmnum{1}_1&=&C\sum_{k=0}^{\infty}\sum_{m=2^{k}}^{2^{k+1}-1}m^{\alpha p-1-\alpha}l(m)\ee|Y|\left(1-\tilde{g}_{\mu}\left(\frac{|Y|}{m^{\alpha}}\right)\right)\\
			&\le&C\sum_{k=1}^{\infty}2^{\alpha pk-\alpha k}l(2^k)\ee|Y|\left(1-\tilde{g}_{\mu}\left(\frac{|Y|}{2^{k\alpha}}\right)\right)\\
			&\le&C\sum_{k=1}^{\infty}2^{\alpha pk-\alpha k}l(2^k)\ee^{*}\sum_{j=k}^{\infty}|Y|g_j\left(\frac{|Y|}{2^{j\alpha}}\right)\\
			&=&C\sum_{j=1}^{\infty}\ee^{*}|Y|g_j\left(\frac{|Y|}{2^{j\alpha}}\right)\sum_{k=1}^{j}2^{\alpha pk-\alpha k}l(2^k)\\
			&\le&C\sum_{j=1}^{\infty}2^{\alpha pj}l(2^j)\vv\left\{|Y|>\mu 2^{(j-1)\alpha}\right\}\\
			&\le&C\sum_{m=1}^{\infty}m^{\alpha p-1}l(m)\vv\left\{|Y|>\mu 2^{-1}m^{\alpha}\right\}\le C C_{\vv}\left\{|Y|^pl(|Y|^{1/\alpha})\right\}<\infty.
		\end{eqnarray*}
		If $p=1$, $C_{\vv}\left\{|Y|^{1+\lambda}\right\}<\infty$ yields $C_{\vv}\left\{|Y|^{1+\lambda'}l(|Y|^{1/\alpha})\right\}<\infty$ for any $0<\lambda'<\lambda$, then by Lemma 2.3 of Xu \cite{Xu2024}, we see that 
		\begin{eqnarray*}
			\Rmnum{1}_1&\le&C\sum_{m=1}^{\infty}m^{-1}\ee|Y|\left(1-\tilde{g}_{\mu}\left(\frac{|Y|}{m^{\alpha}}\right)\right)\sum_{n=1}^{m}n^{-1}l(n)\\
			&\le&C\sum_{m=1}^{\infty}m^{-1}\ee|Y|\left(1-\tilde{g}_{\mu}\left(\frac{|Y|}{m^{\alpha}}\right)\right)\sum_{n=1}^{m}n^{-1+\alpha \lambda'}l(n)\\
			&\le&C\sum_{m=1}^{\infty}m^{\alpha\lambda'-1}l(n)\ee|Y|\left(1-\tilde{g}_{\mu}\left(\frac{|Y|}{m^{\alpha}}\right)\right)
		\end{eqnarray*}
		\begin{eqnarray*}	&=&C\sum_{k=0}^{\infty}\sum_{m=2^{k}}^{2^{k+1}-1}m^{\alpha\lambda'-1}l(m)\ee|Y|\left(1-\tilde{g}_{\mu}\left(\frac{|Y|}{m^{\alpha}}\right)\right)\\
			&\le&C\sum_{k=1}^{\infty}2^{k(\alpha\lambda')}l(2^k)\ee|Y|\left(1-\tilde{g}_{\mu}\left(\frac{|Y|}{2^{k\alpha}}\right)\right)\\
			&\le&C\sum_{k=1}^{\infty}2^{k(\alpha\lambda')}l(2^k)\ee^{*}\left(|Y|\sum_{l=k}^{\infty}g_l\left(\frac{|Y|}{2^{l\alpha}}\right)\right)\\
			&\le&C\sum_{l=1}^{\infty}\ee^{*}\left(|Y|g_l\left(\frac{|Y|}{2^{l\alpha}}\right)\right)\sum_{k=1}^{l}2^{k(\alpha\lambda')}l(2^k)\\
			&\le&C\sum_{l=1}^{\infty}\ee\left(|Y|g_l\left(\frac{|Y|}{2^{l\alpha}}\right)\right)2^{l(\alpha\lambda')}l(2^l)\\
			&\le&C\sum_{l=1}^{\infty}\vv\left\{|Y|>\mu 2^{(l-1)\alpha}\right\}2^{l\alpha(\lambda'+1)}l(2^l)<\infty.
		\end{eqnarray*}
		Hence, we get 
		\begin{equation}\label{3.3}
			\Rmnum{1}_1<\infty.
		\end{equation}
		Next we establish $\Rmnum{1}_2<\infty$. From Markov's inequality under sub-linear expectations, H\"{o}lder's inequality and Lemma \ref{lem03}, follows that
		\begin{eqnarray}\label{3.4}
			\nonumber	\Rmnum{1}_2&\le&C\sum_{n=1}^{\infty}f(n)\int_{ n^{\alpha}}^{\infty}x^{-r}\ee^{*}\left|\sum_{i=-\infty}^{\infty}a_i\sum_{j=i+1}^{i+n}Y_{xj}^{(1)}\right|^r\dif x\\
			\nonumber&\le&C\sum_{n=1}^{\infty}f(n)\int_{ n^{\alpha}}^{\infty}x^{-r}\ee^{*}\left|\sum_{i=-\infty}^{\infty}|a_i|^{\frac{r-1}{r}}\left(|a_i|^{\frac{1}{r}}\left|\sum_{j=i+1}^{i+n}Y_{xj}^{(1)}\right|\right)\right|^r\dif x\\
			\nonumber&\le&C\sum_{n=1}^{\infty}f(n)\int_{ n^{\alpha}}^{\infty}x^{-r}\left(\sum_{i=-\infty}^{\infty}|a_i|\right)^{r-1}\left(\sum_{i=-\infty}^{\infty}|a_i|\ee^{*}\left|\sum_{j=i+1}^{i+n}Y_{xj}^{(1)}\right|^r\right)\dif x\\
			\nonumber&\le&C\sum_{n=1}^{\infty}f(n)\int_{ n^{\alpha}}^{\infty}x^{-r}\left(\sum_{i=-\infty}^{\infty}|a_i|\ee\left|\sum_{j=i+1}^{i+n}Y_{xj}^{(1)}\right|^r\right)\dif x\\
			\nonumber&\le&C\sum_{n=1}^{\infty}f(n)\int_{ n^{\alpha}}^{\infty}x^{-r}\sum_{i=-\infty}^{\infty}|a_i|\sum_{j=i+1}^{i+n}C_{\vv}\left\{|Y_{xj}^{(1)}|^r\right\}\dif x\\
			\nonumber&&+C\sum_{n=1}^{\infty}f(n)g(n)\int_{ n^{\alpha}}^{\infty}x^{-r}\sum_{i=-\infty}^{\infty}|a_i|\left(\sum_{j=i+1}^{i+n}\ee\left|Y_{xj}^{(1)}\right|^2\right)^{r/2}\dif x\\
			\nonumber&&+C\sum_{n=1}^{\infty}f(n)\int_{ n^{\alpha}}^{\infty}x^{-r}\sum_{i=-\infty}^{\infty}|a_i|\left(\sum_{j=i+1}^{i+n}\left|\ee(Y_{xj}^{(1)})\right|+\left|\ee(-Y_{xj}^{(1)})\right|\right)^{r}\dif x\\
			&=&:\Rmnum{1}_{21}+\Rmnum{1}_{22}+\Rmnum{1}_{23},
		\end{eqnarray}
		where $r\ge2$ is given later.
		
		For $\Rmnum{1}_21$, if $p>1$, taking $r>\max\{2,p\}$, then by  $C_r$ inequality, similar proof of (2.8) of Zhang \cite{Zhang2021}, Lemma 2.3 of Xu \cite{Xu2024}, we see that
		\begin{eqnarray}\label{3.5}
			\nonumber	\Rmnum{1}_{21}&\le&C\sum_{n=1}^{\infty}f(n)\int_{ n^{\alpha}}^{\infty}x^{-r}\sum_{i=-\infty}^{\infty}|a_i|\sum_{j=i+1}^{i+n}C_{\vv}\left\{|Y_{x}^{(1)}|^r\right\}\dif x\\
			\nonumber&\le &C\sum_{n=1}^{\infty}f(n)n\sum_{m=n}^{\infty}\int_{ m^{\alpha}}^{(m+1)^{\alpha}}x^{-r}C_{\vv}\left\{|Y_{x}^{(1)}|^r\right\}\dif x\\
			\nonumber&\le &C\sum_{n=1}^{\infty}f(n)n\sum_{m=n}^{\infty}m^{\alpha(1-r)-1}C_{\vv}\left\{|Y_{(m+1)^{\alpha}}^{(1)}|^r\right\}\\
			\nonumber&\le &C\sum_{m=1}^{\infty}m^{\alpha(1-r)-1}\int_{0}^{(m+1)^{\alpha}}\vv\left(|Y|>x\right)x^{r-1}\dif x\sum_{n=1}^{m}f(n)n\\
			\nonumber&\le &C\sum_{m=1}^{\infty}m^{\alpha(1-r)-1}\int_{0}^{(m+1)^{\alpha}}\vv\left(|Y|>x\right)x^{r-1}\dif x m^{\alpha p-\alpha}l(m)\\
			\nonumber&\le &C\sum_{m=1}^{\infty}m^{\alpha(p-r)-1}l(m)\sum_{k=1}^{m+1}\int_{(k-1)^{\alpha}}^{k^{\alpha}}\vv\left(|Y|>x\right)x^{r-1}\dif x \\
			\nonumber&\le &C\sum_{k=1}^{\infty}\int_{(k-1)^{\alpha}}^{k^{\alpha}}\vv\left(|Y|>x\right)x^{r-1}\dif x\sum_{m=1\bigvee(k-1)}^{\infty}m^{\alpha(p-r)-1}l(m) \\
			\nonumber&\le &C\sum_{k=2}^{\infty}\vv\left(|Y|>(k-1)^{\alpha}\right)k^{r\alpha-1}k^{\alpha(p-r)}l(k)+C \sum_{m=1}^{\infty}m^{\alpha(p-r)-1}l(m)\\
			&\le &C\sum_{k=2}^{\infty}\vv\left(|Y|>(1/2)^{\alpha}k^{\alpha}\right)k^{\alpha p-1}l(k)+C<\infty.
		\end{eqnarray}
		For $\Rmnum{1}_{21}$, if $p=1$, taking $r>\max\{1+\lambda',2\}$, where $0<\lambda'<\lambda$, then by the similar discussion as above, we see that
		\begin{eqnarray}\label{3.6}
			\Rmnum{1}_{21}&\le&C\sum_{m=1}^{\infty}m^{\alpha(1-r)-1}\int_{0}^{(m+1)^{\alpha}}\vv\left(|Y|>x\right)x^{r-1}\dif x\sum_{n=1}^{m}f(n)n\\
			\nonumber&\le &C\sum_{m=1}^{\infty}m^{\alpha(1-r+\lambda')-1}\int_{0}^{(m+1)^{\alpha}}\vv\left(|Y|>x\right)x^{r-1}\dif x l(m)\\
			\nonumber&\le &C\sum_{m=1}^{\infty}m^{\alpha(1-r+\lambda')-1}l(m)\sum_{k=1}^{m+1}\int_{(k-1)^{\alpha}}^{k^{\alpha}}\vv\left(|Y|>x\right)x^{r-1}\dif x \\
			\nonumber&\le &C\sum_{k=1}^{\infty}\int_{(k-1)^{\alpha}}^{k^{\alpha}}\vv\left(|Y|>x\right)x^{r-1}\dif x\sum_{m=1\bigvee(k-1)}^{\infty}m^{\alpha(1-r+\lambda')-1}l(m) \\
			\nonumber&\le &C\sum_{k=2}^{\infty}\vv\left(|Y|>(k-1)^{\alpha}\right)k^{r\alpha-1}k^{\alpha(1-r+\lambda')}l(k)+C \sum_{m=1}^{\infty}m^{\alpha(1-r+\lambda')-1}l(m)\\
			&\le &C\sum_{k=2}^{\infty}\vv\left(|Y|>(1/2)^{\alpha}k^{\alpha}\right)k^{\alpha (1+\lambda')-1}l(k)+C<\infty.
		\end{eqnarray}
		For $\Rmnum{1}_{22}$, if $1\le p<2$, observing that $g(n)=O(n^{\delta})$, taking $r>2$ fulfilling that $\alpha p+r/2-\alpha pr/2-1+\delta=(\alpha p-1)(1-r/2)+\delta<0$, then by $C_r$ inequality, we conclude that 
		\begin{eqnarray}\label{3.7}
			\nonumber\Rmnum{1}_{22}&\le&C\sum_{n=1}^{\infty}n^{r/2}f(n)g(n)\int_{ n^{\alpha}}^{\infty}x^{-r}\left(\ee\left|Y_{x}^{(1)}\right|^2\right)^{r/2}\dif x\\
			\nonumber&\le&C\sum_{n=1}^{\infty}n^{r/2}f(n)g(n)\sum_{m=n}^{\infty}\int_{ m^{\alpha}}^{(m+1)^{\alpha}}\\
			\nonumber&&\times \left[x^{-r}\left(\ee\left(|Y|^2\tilde{g}_{\mu}\left(\frac{\mu|Y|}{x}\right)\right)\right)^{r/2}+\left(\ee\left(1-\tilde{g}_{\mu}\left(\frac{|Y|}{x}\right)\right)\right)^{r/2}\right]\dif x\\
			\nonumber&\le&C\sum_{m=1}^{\infty}\left[m^{\alpha(1-r)-1}\ee\left(|Y|^2\tilde{g}_{\mu}\left(\frac{\mu|Y|}{(m+1)^{\alpha}}\right)\right)^{r/2}+m^{\alpha-1}\left(\ee\left(1-\tilde{g}_{\mu}\left(\frac{|Y|}{m^{\alpha}}\right)\right)\right)^{r/2}\right]\\
			\nonumber&&\times\sum_{n=m}^{\infty}n^{r/2}f(n)g(n)\\
			\nonumber&\le&C\sum_{m=1}^{\infty}m^{\alpha(p-r)+r/2+\delta-2}l(m)\left[\ee\left(|Y|^p|Y|^{2-p}\tilde{g}_{\mu}\left(\frac{\mu|Y|}{(m+1)^{\alpha}}\right)\right)\right]^{r/2}\\
			\nonumber&&+C\sum_{m=1}^{\infty}m^{\alpha p+r/2+\delta-2}l(m)\left(\ee\left(1-\tilde{g}_{\mu}\left(\frac{|Y|}{m^{\alpha}}\right)\right)\right)^{r/2}\\
			&\le&C\sum_{m=1}^{\infty}m^{\alpha p+r/2-\alpha pr/2+\delta-2}l(m)\left(\ee|Y|^p\right)^{r/2}<\infty.
		\end{eqnarray}
		For $\Rmnum{1}_{22}$, if $p\ge 2$, observing that $g(n)=O(n^{\delta})$, taking $r>(\alpha p-1)/(\alpha-1/2)\ge p$ satisfying $\alpha(p-r)+r/2+\delta-1<0$, then by $C_r$ inequality, and similar proof of (\ref{3.7}), we have
		\begin{eqnarray}\label{3.8}
			\nonumber	\Rmnum{1}_{22}&\le&C\sum_{m=1}^{\infty}\left[m^{\alpha(1-r)-1}\ee\left(|Y|^2\tilde{g}_{\mu}\left(\frac{\mu|Y|}{(m+1)^{\alpha}}\right)\right)^{r/2}+m^{\alpha-1}\left(\ee\left(1-\tilde{g}_{\mu}\left(\frac{|Y|}{m^{\alpha}}\right)\right)\right)^{r/2}\right]\\
			\nonumber&&\times\sum_{n=m}^{\infty}n^{r/2}f(n)g(n)\\
			\nonumber&\le&C\sum_{m=1}^{\infty}m^{\alpha(p-r)+r/2+\delta-2}l(m)\left(\ee|Y|^2\tilde{g}_{\mu}\left(\frac{\mu|Y|}{(m+1)^{\alpha}}\right)\right)^{r/2}\\
			\nonumber&&+C\sum_{m=1}^{\infty}m^{\alpha p+r/2+\delta-2}l(m)\left(\ee\left(1-\tilde{g}_{\mu}\left(\frac{|Y|}{m^{\alpha}}\right)\right)\right)^{r/2}\\
			&\le&C\sum_{m=1}^{\infty}m^{\alpha(p-r)+r/2+\delta-2}l(m)\left(\ee|Y|^2\right)^{r/2}<\infty.
		\end{eqnarray}
		For $\Rmnum{1}_{23}$, we take $r>2$. By $\ee(Y_i)=\ee(-Y_i)=0$, Proposition 1.3.7 of Peng \cite{Peng2019} and Lemma 4.5 (\rmnum{3}) of Zhang \cite{Zhang2016a}, we get 
		\begin{eqnarray}\label{3.9}
			\nonumber\Rmnum{1}_{23}&\le&C\sum_{n=1}^{\infty}f(n)\sum_{m=n}^{\infty}\int_{ m^{\alpha}}^{(m+1)^{\alpha}}x^{-r}\left(\sup_{i}\sum_{j=i+1}^{i+n}\ee\left|Y_{xj}^{(1)}-Y_j\right|\right)^{r}\dif x\\
			\nonumber&\le&C\sum_{n=1}^{\infty}f(n)\sum_{m=n}^{\infty}\int_{ m^{\alpha}}^{(m+1)^{\alpha}}x^{-r}n^r\left(\ee\left|Y\right|\left(1-\tilde{g}_{\mu}\left(\frac{|Y|}{x}\right)\right)\right)^{r}\dif x
		\end{eqnarray}
		\begin{eqnarray}\label{3.9}
			\nonumber&\le&C\sum_{n=1}^{\infty}f(n)n^r\sum_{m=n}^{\infty}m^{\alpha(1-r)-1}\left(\ee\left|Y\right|\left(1-\tilde{g}_{\mu}\left(\frac{|Y|}{m^{\alpha}}\right)\right)\right)^{r}\\
			\nonumber&\le&C\sum_{m=1}^{\infty}m^{\alpha(1-r)-1}\frac{\ee\left(|Y|^pl(|Y|^{1/\alpha})\right)^r}{m^{\alpha(p-1)r}l^r(m)}\sum_{n=1}^{m}f(n)n^r\\
			&\le&C\sum_{m=1}^{\infty}m^{-(\alpha p-1)(r-1)-1}/l^r(m)\left(C_{\vv}\left\{|Y|^pl(|Y|^{1/\alpha})\right\}\right)^r<\infty.
		\end{eqnarray}
		Hence, (\ref{3.1}) is established by (\ref{3.2})-(\ref{3.9}).
	\end{proof}
	We next investigate the case $\alpha p=1$.
	\begin{thm}\label{thm02}
		Assume that $l$ is a function slowly varying at infinity, $1\le p<2$. Suppose that $\sum_{i=-\infty}^{\infty}|a_i|^{\theta}<\infty$, where $\theta\in (0,1)$ if $p=1$ and $\theta=1$ if $1<p<2$. Assume that $\{X_n=\sum_{i=-\infty}^{\infty}a_iY_{i+n},n\ge 1\}$ is a moving average process produced by a sequence $\{Y_i,-\infty<i<\infty\}$ of $m$-WA random variables with dominating $g(n)=O(n^{\delta})$ for some $0\le \delta<(2-p)/p$, stochastically dominated by a random variable $Y$. While $p=1$, assume that $0<\delta<1$. If $\ee(Y_i)=\ee(-Y_i)=0$ and $C_{\vv}\left\{|Y|^{p(1+\delta)}l(|Y|^p)\right\}<\infty$, then for any $\varepsilon>0$,
		\begin{equation}\label{3.10}
			\sum_{n=1}^{\infty}n^{-1-1/p}l(n)C_{\vv}\left\{\left(\left|\sum_{j=1}^{k}X_j\right|-\varepsilon n^{1/p}\right)^{+}\right\}<\infty.
		\end{equation}
	\end{thm}
	\begin{proof}
		Let $h(x)=n^{-1-1/p}l(n)$. As in the proof of (\ref{3.2}), we have 
		\begin{eqnarray}\label{3.11}
			\nonumber&&	\sum_{n=1}^{\infty}h(n)C_{\vv}\left\{\left(\left|\sum_{j=1}^{k}X_j\right|-\varepsilon n^{1/p}\right)^{+}\right\}\\
			\nonumber&\le&C\sum_{n=1}^{\infty}h(n)\int_{n^{1/p}}^{\infty}\vv\left\{\left|\sum_{i=-\infty}^{\infty}a_i\sum_{j=i+1}^{i+n}Y_{xj}^{(2)}\right|> \varepsilon x/2\right\}\dif x\\
			\nonumber&&+C\sum_{n=1}^{\infty}h(n)\int_{n^{1/p}}^{\infty}\vv\left\{\left|\sum_{i=-\infty}^{\infty}a_i\sum_{j=i+1}^{i+n}Y_{xj}^{(1)}\right|> \varepsilon x/2\right\}\dif x\\
			&=&:J_1+J_2.
		\end{eqnarray}
		For $J_1$, take $\varepsilon'=\delta$. By Markov's inequality under sub-linear expectations, $C_r$ inequality, and Lemma 2.3 of Xu \cite{Xu2024}, we see that
		\begin{eqnarray}\label{3.12}
			\nonumber J_1&\le&C\sum_{n=1}^{\infty}h(n)\int_{n^{1/p}}^{\infty}x^{-\theta}\ee^{*}\left|\sum_{i=-\infty}^{\infty}a_i\sum_{j=i+1}^{i+n}Y_{xj}^{(2)}\right|^2\dif x\\
			\nonumber &\le&C\sum_{n=1}^{\infty}nh(n)\int_{n^{1/p}}^{\infty}x^{-\theta}\ee|Y_{x}^{(2)}|^{\theta}\dif x\\
			\nonumber &\le&C\sum_{n=1}^{\infty}nh(n)\sum_{m=n}^{\infty}\int_{m^{1/p}}^{(m+1)^{1/p}}x^{-\theta}\ee|Y|^{\theta}\left(1-\tilde{g}_{\mu}\left(\frac{|Y|}{x}\right)\right)\dif x
		\end{eqnarray}
		\begin{eqnarray}\label{3.12}
			\nonumber&\le&C\sum_{n=1}^{\infty}nh(n)\sum_{m=n}^{\infty}m^{(1-\theta)/p-1}\ee|Y|^{\theta}\left(1-\tilde{g}_{\mu}\left(\frac{|Y|}{m^{1/p}}\right)\right)\\
			\nonumber&=&C\sum_{m=1}^{\infty}m^{(1-\theta)/p-1}\ee|Y|^{\theta}\left(1-\tilde{g}_{\mu}\left(\frac{|Y|}{m^{1/p}}\right)\right)\sum_{n=1}^{m}nh(n)\\
			\nonumber&\le&\begin{cases}
				C\sum_{m=1}^{\infty}m^{-1/p}l(m)\ee|Y|\left(1-\tilde{g}_{\mu}\left(\frac{|Y|}{m^{1/p}}\right)\right), & \text{  $1<p<2$}\\
				C\sum_{m=1}^{\infty}m^{(1-\theta)/p-1}\ee|Y|^{\theta}\left(1-\tilde{g}_{\mu}\left(\frac{|Y|}{m^{1/p}}\right)\right)\sum_{n=1}^{m}n^{\varepsilon'-1}l(n), & \text{  $p=1$}
			\end{cases}\\
			\nonumber&\le&	\begin{cases}
				C\sum_{k=0}^{\infty}\sum_{m=2^k}^{2^{k+1}-1}m^{-1/p}l(m)\ee|Y|\left(1-\tilde{g}_{\mu}\left(\frac{|Y|}{m^{1/p}}\right)\right), & \text{  $1<p<2$}\\
				C\sum_{k=0}^{\infty}\sum_{m=2^k}^{2^{k+1}-1}m^{-\theta+\varepsilon'}l(m)\ee|Y|^{\theta}\left(1-\tilde{g}_{\mu}\left(\frac{|Y|}{m}\right)\right), & \text{  $p=1$}
			\end{cases}\\
			\nonumber&\le&
			\begin{cases}
				C\sum_{k=1}^{\infty}2^{k(-1/p+1)}l(2^k)\ee^{*}\left(|Y|\sum_{j=k}^{\infty}g_j\left(\frac{|Y|}{2^{j/p}}\right)\right), & \text{  $1<p<2$}\\
				C\sum_{k=1}^{\infty}2^{k(-\theta+\varepsilon'+1)}l(2^k)\ee^{*}\left(|Y|^{\theta}\sum_{j=k}^{\infty}g_j\left(\frac{|Y|}{2^{j}}\right)\right), & \text{  $p=1$}
			\end{cases}\\
			\nonumber&\le&
			\begin{cases}
				C\sum_{j=1}^{\infty}\ee^{*}\left(|Y|g_j\left(\frac{|Y|}{2^{j/p}}\right)\right)\sum_{k=1}^{j}2^{k(-1/p+1)}l(2^k), & \text{  $1<p<2$}\\
				C\sum_{j=1}^{\infty}\ee^{*}\left(|Y|^{\theta}g_j\left(\frac{|Y|}{2^{j}}\right)\right)\sum_{k=1}^{j}2^{k(-\theta+\varepsilon'+1)}l(2^k), & \text{  $p=1$}
			\end{cases}\\
			&\le&
			\begin{cases}
				C\sum_{j=1}^{\infty}\vv\left\{|Y|>\mu 2^{(j-1)/p}\right\}2^{j}l(2^j)<\infty, & \text{  $1<p<2$}\\
				C\sum_{j=1}^{\infty}\vv\left\{|Y|>\mu 2^{(j-1)}\right\}2^{j(\varepsilon'+1)}l(2^k)<\infty, & \text{  $p=1$,}
			\end{cases}
		\end{eqnarray}
		where $\tilde{g}_{\mu}(\cdot)$, $g_j(\cdot)$ here are defined as those of (\ref{001}) and (\ref{002}) with only $1/p$ here in place of $\alpha$ there.
		
		For $J_2$, as in the proof of $\Rmnum{1}_2$, observing that $g(n)=O(n^{\delta})$, for some $0\le\delta<(2-p)/p $, taking $q=2$ by Lemma \ref{lem03} and similar proof of (2.8) of Zhang \cite{Zhang2021}, we get
		\begin{eqnarray}\label{3.13}
			\nonumber J_2&\le&C\sum_{n=1}^{\infty}h(n)\int_{n^{1/p}}^{\infty}x^{-2}\ee^{*}\left|\sum_{i=-\infty}^{\infty}a_i\sum_{j=i+1}^{i+n}Y_{xj}^{(1)}\right|^2\dif x\\
			\nonumber&\le&C\sum_{n=1}^{\infty}nh(n)(1+g(n))\sum_{m=n}^{\infty}\int_{m^{1/p}}^{(m+1)^{1/p}}x^{-2}C_{\vv}\left\{|Y_{x}^{(1)}|^2\right\}\dif x\\
			\nonumber	&&+C\sum_{n=1}^{\infty}h(n)\sum_{m=n}^{\infty}\int_{m^{1/p}}^{(m+1)^{1/p}}x^{-2}\left[\sum_{i=1}^{n}|\ee(Y_{xj}^{(1)})|+|\ee(-Y_{xj}^{(1)})|\right]^2\dif x\\
			\nonumber&\le&C\sum_{n=1}^{\infty}nh(n)(1+g(n))\sum_{m=n}^{\infty}m^{-1/p-1}\int_{0}^{(m+1)^{2/p}}\vv\left\{|Y|^2>y\right\}\dif y\\
			\nonumber&&+C\sum_{n=1}^{\infty}h(n)\sum_{m=n}^{\infty}m^{-1/p-1}\left[\sum_{i=1}^{n}\ee(|Y_{m^{1/p}j}^{(2)}|)\right]^2\\
			\nonumber&\le&C\sum_{m=1}^{\infty}m^{-1/p-1}\int_{0}^{(m+1)^{2/p}}\vv\left\{|Y|^2>y\right\}\dif y\sum_{n=1}^{m}n^{-1/p}l(n)(1+g(n))\\
			\nonumber&&+C\sum_{m=1}^{\infty}m^{-1/p-1}\left[\ee(|Y_{m^{1/p}}^{(2)}|)\right]^2\sum_{n=1}^{m}n^{1-1/p}l(n)
		\end{eqnarray}
		\begin{eqnarray}\label{3.13}	
			\nonumber&\le&C\sum_{m=1}^{\infty}m^{-2/p+\delta}l(m)\int_{0}^{(m+1)^{2/p}}\vv\left\{|Y|^2>y\right\}\dif y\\
			\nonumber&&+C\sum_{m=1}^{\infty}m^{-2/p+1}l(m)\left[\ee(|Y_{m^{1/p}}^{(2)}|)\right]^2\\	\nonumber&\le&C\sum_{m=1}^{\infty}m^{-2/p+\delta}l(m)\sum_{\ell=1}^{m+1}\int_{(\ell-1)^{2/p}}^{(\ell)^{2/p}}\vv\left\{|Y|^2>y\right\}\dif y\\
			\nonumber&&+C\sum_{m=1}^{\infty}m^{-2/p+1}l(m)\left[\ee|Y|\left(1-\tilde{g}_{\mu}\left(\frac{|Y|}{m^{1/p}}\right)\right)\right]^2\\
			\nonumber&\le&C\sum_{\ell=1}^{\infty}\int_{(\ell-1)^{2/p}}^{(\ell)^{2/p}}\vv\left\{|Y|^2>y\right\}\dif y\sum_{m=1\bigvee (\ell-1)}^{\infty}m^{-2/p+\delta}l(m)\\
			\nonumber&&+C\sum_{m=1}^{\infty}m^{-2/p+1}l(m)\frac{\left(\ee\left(|Y|^{p(1+\delta)}l(|Y|^p)\right)\right)^2}{\left(m^{(p(1+\delta)-1)/p}l(m)\right)^2}\\
			\nonumber&\le&C\sum_{\ell=2}^{\infty}\int_{(\ell-1)^{2/p}}^{(\ell)^{2/p}}\vv\left\{|Y|^2>y\right\}\ell^{1+\delta-2/p}l(\ell)\dif y+C\\
			\nonumber&&+C\sum_{m=1}^{\infty}m^{-2\delta-1}/l(m)\left(C_{\vv}\left\{|Y|^{p(1+\delta)}l(|Y|^p)\right\}\right)^2\\
			\nonumber&\le&C\sum_{\ell=2}^{\infty}\int_{(\ell-1)^{2/p}}^{\ell^{2/p}}\vv\left\{|Y|^2>y\right\}y^{(1+\delta-2/p)p/2}l(y^{p/2})\dif y+C\\
			&\le&C\int_{1}^{\infty}\vv\left\{|Y|^p>y\right\}\dif (l(y)y^{1+\delta})+C\le CC_{\vv}\left\{|Y|^{p(1+\delta)}l(|Y|^p)\right\}<\infty.
		\end{eqnarray}
		Therefore, combining (\ref{3.11})-(\ref{3.13}) results in (\ref{3.10}). This completes the proof.
	\end{proof}
	By Theorems \ref{thm01}, \ref{thm02}, we conclude that the following Corollary holds.
	\begin{cor}
		Under the same conditions of Theorem \ref{thm01}, for any $\epsilon>0$, we have
		\begin{eqnarray}\label{3.14}
			\sum_{n=1}^{\infty}n^{\alpha p-2}l(n)\vv\left\{\left|\sum_{j=1}^{n}X_j\right|>\epsilon n^{\alpha}\right\}<\infty.
		\end{eqnarray}
		Under the conditions of Theorem \ref{thm02}, for any $\epsilon>0$, we have
		\begin{eqnarray}\label{3.15}
			\sum_{n=1}^{\infty}n^{-1}l(n)\vv\left\{\left|\sum_{j=1}^{n}X_j\right|>\epsilon n^{1/p}\right\}<\infty.
		\end{eqnarray}
	\end{cor}

	{\bf Acknowledgements}
	
	Not applicable.
	
	{\bf Funding}
	
	This research was supported by Science and Technology Research Project of Jiangxi Provincial Department of Education of China (No. GJJ2201041), Doctoral Scientific Research Starting Foundation of Jingdezhen Ceramic University ( No. 102/01003002031), Re-accompanying Funding Project of Academic Achievements of Jingdezhen Ceramic University (No. 215/20506277).
	
	{\bf Availability of data and materials}
	
	No data were used to support this study.
	
	{\bf Competing interests}
	
	The authors declare that they have no competing interests.
	
	{\bf Authors contributions}
	
	Both authors contributed equally and read and approved the final manuscript.

\end{document}